\documentclass[runningheads, envcountsame, smallextended]{llncs}
\usepackage[T1]{fontenc}
\usepackage{amsmath,amssymb}
\usepackage{graphicx}
\usepackage{enumitem}
\usepackage{comment}
\usepackage{url}

\newcommand{\vol}{{\rm vol}}
\newcommand{\pr}{{\mathbb P}}
\newcommand{\q}{q^*}
\newcommand{\cA}{{\mathcal A}}

\numberwithin{theorem}{section}
\numberwithin{definition}{section}
\numberwithin{lemma}{section}
\numberwithin{corollary}{section}
\numberwithin{proposition}{section}
\numberwithin{conjecture}{section}
\numberwithin{remark}{section}
\numberwithin{equation}{section}
\begin{document}
\title{Note on edge expansion and modularity in preferential attachment graphs}
\author{Colin McDiarmid\inst{1} \and
Katarzyna Rybarczyk\inst{2} \and
Fiona Skerman\inst{3} \and Ma\l gorzata~Sulkowska\inst{4}}
\authorrunning{C.~McDiarmid, K.~Rybarczyk, 
F.~Skerman and M.~Sulkowska}
\institute{Department of Statistics, University of Oxford, UK \email{cmcd@stats.ox.ac.uk}\and
Adam Mickiewicz University, Pozna\' n, Poland \email{kryba@amu.edu.pl}
\and Department of Mathematics, Uppsala University, Sweden
 \email{fiona.skerman@math.uu.se}
 \and Department of Fundamentals of Computer Science, Wroc{\l}aw University of Science and Technology, Poland \email{malgorzata.sulkowska@pwr.edu.pl}}
\maketitle              
\begin{abstract}
	Edge expansion is a parameter indicating how well-connected a graph is. It is useful for designing robust networks, analysing random walks or information flow through a network and is an important notion in theoretical computer science.   
    Modularity is a measure of how well a graph can be partitioned into communities and is widely used in clustering applications. We study these two parameters in two commonly considered models of random preferential attachment graphs, with~$h\geq 2$ edges added per step. 
    We establish new bounds for the likely edge expansion for both random models. 
    Using bounds for edge expansion of small subsets of vertices, we derive new upper bounds also for the modularity values for small $h$.

\end{abstract}
\section{Introduction and statement of results}  
\subsection{Introduction} 
"Preferential attachment", often referred to as the "rich get richer" effect, describes a pattern in the growth of dynamic graph models where new nodes  tend to form links with nodes that already have numerous connections. 	
The earliest random graphs incorporating the idea of preferential attachment were random recursive trees \cite{Szymanski1993,Szymanski1987}. 
However, it was not until the results of  Barab{\'a}si and Albert \cite{BA_basic} that the connection between the phenomenon of "preferential attachment" and certain natural properties of networks was demonstrated, initiating more in-depth research into various models based on this property.  
In this paper, we focus on two fundamental models of such graphs, $G_n^h$ and $\tilde{G}_n^h$. They are generated from a finite initial graph by adding new vertices one at a time. Each new vertex connects to $h\ge 2$ already existing vertices, and these are chosen with probability proportional to their current degrees (for the exact definitions consult Section~\ref{sec:models}). For generalised preferential attachment models -- see~\cite{hazra2023percolation,Hofstad_2017}.

We investigate two parameters that describe certain important characteristics of networks: edge expansion and modularity. Edge expansion (also known as isoperimetric number, Cheeger constant or, simply, expansion), denoted by $\alpha(G)$, is a parameter indicating how well-connected a graph $G$ is.
It lower bounds the ratio of the number of edges leaving a set $S$ of vertices to the size of $S$ (see Definition~\ref{def:uexpansion}).  Edge expansion is particularly important in the study of network robustness or efficiency of communication in complex systems and also in the theory of error correcting codes~\cite{HoLiWi06}. 

The second graph parameter we consider is modularity, denoted by $q^*(G)$. It was introduced in~2004 by Newman and Girvan~\cite{NeGi04} as a measure of how well a graph can be partitioned into communities. Roughly speaking, a vertex-partition with high modularity score has many more internal edges than we would expect if the edges were re-wired randomly, matching the original degree sequence of the graph (see Definition~\ref{def:modularity}). The modularity score of a partition is used in clustering algorithms popular in applications~\cite{BlGuLaLe08,KaPrTh21,NeBook18,TrWaEc19}. Early theoretical research on modularity was conducted for trees \cite{Bagrow_trees_12} and random regular graphs and lattices~\cite{McSk_reg_lattice_13}.  
Results for the basic and most studied random graph, binomial $G(n,p)$, were given in~\cite{DiSk20} and recently partially extended in~\cite{RybarczykSulkowska_Gnp}. See the final section of~\cite{mcdiarmid_skerman_2025_chapter} for a summary of known modularity values, and for further recent work consult~\cite{LaSu23,LiMi22,modexpansion,Rybarczyk2025randomintersectiongraphs}.

First, we revisit the result of Mihail, Papadimitriou and Saberi~\cite{mihail2006journal} who proved that for large $h$ the expansion of $\tilde{G}_n^h$ is greater than $1/5$ with high probability (whp) -- that is, with probability tending to one as $n \to \infty$. 
We show that, with slight modifications of their proof, one may obtain a whp lower bound on expansion (and also expansion of small sets), in both $G_n^h$ and $\tilde{G}_n^h$, which is linear in $h$, see Theorem~\ref{thm:our_u_expansion_fussy} and Corollary~\ref{cor:exp_linear_in_h}. 
(Note that similar linear dependence on $h$ for $\tilde{G}_n^h$ was claimed, but not proved, in the conference version of the paper by Mihail et~al.~\cite{mihail2003conf}. The journal version of the Mihail et.~al.~paper, i.e.~\cite{mihail2006journal} contains the full proof but the dependence is not linear in $h$, see Theorem~\ref{thm:mihail_expansion} below.) In a related direction, Hazra, van der Hofstad and Ray~\cite{hazra2023percolation} studied expansion of large sets for very general classes of preferential attachment graphs (including both $G_n^h$ and $\tilde{G}_n^h$) though they also did not obtain linear dependence in $h$.

By considering our new expansion result for small sets, we improve the whp upper bound for $q^*(\tilde{G}_n^h)$ obtained by Prokhorenkova, Pra\l at and Raigorodskii~\cite{prokhorenkova2017modularity_internet_Math}, and show that our bound extends to $G_n^h$ also (see Corollary~\ref{cor.q_upper_bound}). The previous best upper bound for fixed $h$ (from~\cite{prokhorenkova2017modularity_internet_Math}) is $\q(\tilde{G}_n^h) \leq 1-1/10h$ (see Section~\ref{sec:modularity} and Remark~\ref{rem:result_really_is} for details). Our methods give the whp upper bound on $\q(\tilde{G}_n^h)$ and $\q(G_n^h)$ of $0.92383$, which is independent of~$h$ and thus whp bounds the modularity values away from $1$ for all $h \geq 2$.
 
 As a byproduct of our proofs we also obtain a general upper bound on modularity value in terms of expansion of small sets, see 
 Theorem~\ref{thm:general_upper_bound_mod_2}.

The second and fourth author proved recently in~\cite{rybarczyk2025modularity} that $q^*(G_n^h)$ may be whp upper bounded by a function of $h$ which tends to $0$ as $h \rightarrow \infty$ -- see Theorem~\ref{thm.sqrth}. However, for $2 \leq h \leq 959$ 
the upper bound obtained in this paper is tighter. 
We recall that for practical applications the average degree is often relatively small, so the results in this paper fill an important gap.

\subsection{Basic definitions}

All the graphs considered in this paper are finite, and loops and multiple edges are allowed (thus they are multigraphs). For any set $S$, $S^{(k)}$ stands for the set of all $k$-element subsets of $S$. A graph is a pair $G=(V,E)$, where $V$ is a finite set of vertices and $E$ is a finite multiset of elements from $V^{(1)} \cup V^{(2)}$. For $S \subseteq V$ let $\bar{S} = V \setminus S$.  For $S,U \subseteq V$ set $E_G(S) = \{e \in E \cap (S^{(1)} \cup S^{(2)})\}$ (the set of inner edges of $S$) and $E_G(S,U) = \{e \in E: e \cap S \neq \emptyset \wedge e \cap U \neq \emptyset\}$ (the set of edges between $S$ and $U$). Denote also $e(G) = |E|$, $e_G(S) = |E_G(S)|$ and $e_G(S,U) = |E_G(S,U)|$. The degree of a vertex $v \in V$ in $G$, denoted by $\deg_G(v)$, is the number of edges to which $v$ belongs but loops are counted twice, i.e., $\deg_G(v) = 2 |\{e \in E: v \in e \wedge e \in V^{(1)}\}| + |\{e \in E: v \in e \wedge e \in  V^{(2)}\}|$. In some special cases a loop may contribute only $1$ to the degree of a vertex but this will be always clearly indicated. Next, we define the volume of $S \subseteq V$ in $G$ by $\vol_G(S) = \sum_{v \in S} \deg_G(v)$, and write $\vol(G)$ for $\vol_G(V)$. Whenever the context is clear we omit the subscript $G$ and write, e.g., $e(S)$ instead of $e_G(S)$.\\

We focus on two important graph parameters: edge expansion (also in a version restricted to small sets) and modularity. For connections between modularity and conductance, a parameter similar to edge expansion, consult~\cite{modexpansion}.

\begin{definition}[Expansion and $u$-bounded expansion] \label{def:uexpansion} Let $G=(V,E)$ and $0 < u \leq 1/2$. The \emph{$u$-bounded edge expansion} (or $u$-\emph{expansion}) of $G$ is given by
	\[
	\alpha_u(G) = \min_{\substack{S \subseteq V \\ 1 \leq |S| \leq u|V|}} \frac{e(S,\bar{S})}{|S|}
	\]
	except that if $u|V|<1$ we set $\alpha_u(G)=\infty$. 
	Set the \emph{expansion} of a graph $G$ to be~$\alpha(G)=\alpha_{1/2}(G)$.
\end{definition}

\begin{definition}[Modularity, \cite{NeGi04}] \label{def:modularity}
	Let $G=(V,E)$ be a graph with at least one edge. For a partition $\mathcal{A}$ of $V$ define the
    modularity score of $G$ to be
	\[
	q_{\mathcal{A}}(G) = \sum_{S\in\mathcal{A}}\left(\frac{e(S)}{e(G)}-\left(\frac{\vol(S)}{\vol(G)}\right)^2\right)
    .
	\]
	The modularity of $G$ is given by $\q(G) = \max_{\mathcal{A}}q_{\mathcal{A}}(G)$,
	where the maximum is over all partitions $\cA$ of $V$. For graphs $G$ with no edges, for each partition~$\cA$ we set $q_\cA(G)=\q(G)=0$.
\end{definition}

\subsection{The preferential attachment models}\label{sec:models}

The results stated in this paper refer to two commonly considered versions of preferential attachment graphs.\\

\noindent\textbf{Model $\mathbf{G_n^h}$.} Given $h,n \in \mathbb{N}$, we sample the random preferential attachment graph $G_n^h$ as follows. Let $T_1$ be the graph consisting of a single mini-vertex $1$ with a loop $e_1$ (thus the degree of vertex $1$ is $2$). For $t \geq 1$, the graph $T_{t+1}$ is built upon $T_{t}$ by adding a mini-vertex $t+1$, and joining it by an edge $e_{t+1}$ to a random mini-vertex $i$ according to
\[
\pr(i=s) = \begin{cases}  \frac{\deg_{T_t}(s)}{2t+1} & \textnormal{for} \quad 1 \leq s \leq t  \\
	\frac{1}{2t+1} & \textnormal{for} \quad s=t+1.
\end{cases}
\]
We continue the process to obtain the random tree $T_{hn}$. (We call $T_{hn}$ a tree, even though $T_{hn}$ is not necessarily connected, and loops, i.e. single-vertex edges, are allowed in $T_{hn}$.) Next, a random multigraph $G_n^h$ is obtained from $T_{hn}$ by merging each set of mini-vertices $\{h(i-1)+1, h(i-1)+2, \ldots, h(i-1)+h\}$ 
to a single vertex $i$ for $i \in \{1,2, \ldots, n\}$, keeping loops and multiple edges.
This definition of the preferential attachment graph can be found in \cite{Bol_Rio_chapter,Bol_Rio_diameter,Bollobas2001_deg_seq,Hofstad_2017}.\\

\noindent\textbf{Model $\mathbf{\tilde{G}_n^h}$.} The following similar model, which we shall denote by $\tilde{G}_n^h$, was analyzed in \cite{mihail2003conf,mihail2006journal} (here also $h,n \in \mathbb{N}$). Let $\tilde{T}_1$ be again a graph consisting of a single mini-vertex $1$ with a loop $e_1$, but this time $e_1$ contributes just $1$ to the degree of mini-vertex $1$. For $t \geq 1$, the graph $\tilde{T}_{t+1}$ is built upon $\tilde{T}_{t}$, similarly to above, with 
\[
\pr(i=s) = \frac{\deg_{\tilde{T}_t}(s)}{2t-1}  \quad \textnormal{for} \quad 1 \leq s \leq t.
\]
A random multigraph $\tilde{G}_n^h$ is obtained from $\tilde{T}_{hn}$ by a merging analogous to the one described for~$G_n^h$. 
Each loop, apart from $e_1$, contributes $2$ to the degree of a vertex (the loop $e_1$ contributes $1$ to the degree of vertex 1).\\

Observe that $G_n^h$ has average degree $2h$, and $\tilde{G}_n^h$ has average degree close to $2h$. One may see that $\tilde{G}_n^h$ is always connected, while $G_n^h$ is not necessarily connected (however, by Theorem 1 in \cite{Bol_Rio_diameter} we know that for $h\geq 2$ it is connected whp).
Moreover, $\tilde{T}_{hn}$ does not feature loops apart from $e_1$, while $T_{hn}$ allows for a loop at each time step. In $\tilde{G}_n^h$ there is one special loop $e_1$, which contributes only $1$ to the vertex degree (and there may be other loops).

For $G$ sampled according to the above random processes we write $G \sim G_n^h$ and $G \sim \tilde{G}_n^h$, respectively.

\subsection{Expansion of the preferential attachment models.}\label{sec:expansion}

Mihail, Papadimitriou, and Saberi proved in \cite{mihail2006journal} that whp the edge expansion of $G \sim \tilde{G}_n^h$ is at least a positive constant.

\begin{theorem}[Theorem 1 of \cite{mihail2006journal}] \label{thm:mihail_expansion} 
	Let $h \geq 2$ and $G \sim \tilde{G}_n^h$. Then for any $c$ with     $0<c<2(h-1)-1$ if $\,\hat{\alpha} < \min\left\{ \frac{h-1}{2}-\frac{c+1}{4}, \frac15, \frac{(h-1)\ln 2 -\frac{2}{5}\ln 5}{2(\ln h +\ln 2 +1)}\right\}$ then  
	\[
	\pr\big(\;\alpha(G) \geq \hat{\alpha}\;\big) = 1 - o(n^{-c}).
	\]
\end{theorem}
Note that, for large $h$, this theorem implies a whp lower bound of $1/5$ for the expansion of $G \sim \tilde{G}_n^h$. By Definition~\ref{def:uexpansion}, this implies that for large $h$ and any vertex subset~$S$ with $|S| \le n/2$, the number of outgoing edges from $S$ is at least $|S|/5$. However, one would expect that the denser $\tilde{G}_n^h$ is, the larger the expansion; i.e., $\alpha(\tilde{G}_n^h)$ should increase with $h$.

In this paper we modify the proof from \cite{mihail2006journal}, and obtain a whp lower bound for the expansion of $G \sim \tilde{G}_n^h$ and $G \sim G_n^h$ which is linear in $h$. We also extend their result to edge expansion of small sets (i.e., we present the result for $\alpha_u(G)$, where $0<u\leq 1/2$, not only for $\alpha(G) = \alpha_{1/2}(G)$).

\begin{theorem} \label{thm:our_u_expansion_fussy}
	Let $h\geq 2$, 
    $G\sim G_n^h$ or $G \sim \tilde{G}_n^h$, and $0 < u \leq 1/2$. Let also $0< x \leq 1$	be such that $(e/u x)^{2h x} < (1/u)^{h-1}$. If 
    $\hat{\alpha}_u < \min\{ (h-1)/2, hx\}$, then there exists a constant $c>0$ such that
	\[
	\pr\big(\;\alpha_u(G) \geq \hat{\alpha}_u \;\big) = 1 - o(n^{-c}).
	\]
\end{theorem}

\begin{corollary} \label{cor:exp_linear_in_h}
	Let $h \geq 2%3
    $ and $G \sim G_n^h$ or $G \sim \tilde{G}_n^h$. Then there exists a constant $c>0$ such that 
	\[
	\pr\big(\;\alpha(G) \geq 0.03418 \, h \;\big) \geq 1-o(n^{-c}).
	\]
\end{corollary}
Therefore in the models we consider, for each $h\ge 2$, whp for each non-empty vertex subset $S$ with $|S|\le n/2$ we have that $e(S,\bar S)\ge |S|\,0.03418 \, h$ -- see Definition~\ref{def:uexpansion}.

Corollary~\ref{cor:exp_linear_in_h} gives a lower bound on expansion which is linear in $h$, thus improving the journal version of the result by~\cite{mihail2006journal}. Note that up to constants we obtain the bound which was claimed in Theorem 2.1 of~\cite{mihail2003conf}, 
but which was not proven. The claimed bound was that $\alpha(G)\geq \eta $ whp for any $\eta<h/2-3/4$.

The proofs of Theorem~\ref{thm:our_u_expansion_fussy} and Corollary~\ref{cor:exp_linear_in_h} are given in Section~\ref{sec:proofs_exp}.

\subsection{Modularity of the preferential attachment models.}\label{sec:modularity}

Prokhorenkova, Pra\l at and Raigorodskii were the first to establish 
bounds for the modularity of preferential attachment graphs. %
In~\cite{prokhorenkova2017modularity_internet_Math} they showed a whp~$\Omega(h^{-1/2})$ lower bound for $\q(G_n^h)$. 
Later, Agdur, Kam\v{c}ev, and the third author showed that this lower bound holds also for a generalised model of preferential attachment graphs using only likely properties of the graphs degree sequence~\cite{mod2023universal}.

In this paper we investigate upper bounds. The authors of~\cite{prokhorenkova2017modularity_internet_Math} get their upper bound on $q^*(\tilde{G}_n^h)$ by the following proposition -- see also Remark~\ref{rem:result_really_is}.

\begin{proposition} [\cite{prokhorenkova2017modularity_internet_Math}]\label{prop:PPR} 	
Suppose a graph $G$ has edge expansion $\alpha(G)$. Suppose also that $G$ has minimum degree $h \geq 1$ and average degree at most $2h$. Then
\begin{equation}\label{eqn:internet_upperbound}
	\q(G) \leq 1 - \min\left\{\alpha(G)/2h, 1/16\right\}.
\end{equation}
\end{proposition}
We note in the following lemma that $1/16$ above may be replaced by $3/16$ (the proof of the lemma is given in Section~\ref{sec:proofs_mod}).

\begin{lemma}\label{lem:us_upper_bound} Suppose a graph $G$ has edge expansion $\alpha(G)$. 
Suppose also that $G$ has minimum degree $h \geq 1$ and average degree at most $2h$. Then
    \[	
    \q(G) \leq  1 - \min\{\alpha(G)/2h, 3/16\}. 
    \]
\end{lemma}

\begin{remark}\label{rem:result_really_is}
The authors of~\cite{prokhorenkova2017modularity_internet_Math} state that for $G \sim \tilde{G}_n^h$ with $h \geq 3$ whp $q^*(G) \leq 15/16 \approx 0.9375$ and propose a similar bound for $h=2$.
These results, however, were obtained using~\eqref{eqn:internet_upperbound} and some claimed but not proven expansion properties of $\tilde{G}_n^h$ from~\cite{mihail2003conf}. If, instead, one uses the proven expansion properties given in~\cite{mihail2006journal} (Theorem~\ref{thm:mihail_expansion}), the result for $G \sim \tilde{G}_n^h$ becomes whp $q^*(G) \leq 1-1/10h$, thus the upper bound tends to $1$ with $h$ tending to infinity.

Note that applying~\eqref{eqn:internet_upperbound} and our Corollary~\ref{cor:exp_linear_in_h} instead of Theorem~\ref{thm:mihail_expansion} gives for $G \sim \tilde{G}_n^h$ whp $\q(G)\leq 1-0.03418/2 = 0.98291$ thus whp bounds the modularity value of~$G$ away from~$1$ for all $h \geq 2$ (since the lower bound for the edge expansion in Corollary \ref{cor:exp_linear_in_h} is whp \emph{linear} in $h$).
\end{remark}

In this paper we obtain a sharper upper bound for the modularity value, this time generalized to both models, i.e., for $G\sim G^h_n$ and $G\sim \tilde{G}^h_n$. 
We use Theorem~\ref{thm:our_u_expansion_fussy} (which lower bounds the expansion of sets as a function of their size), and consider expansion of small vertex sets, 
to get some general upper bound on the modularity value in the graphs with good expansion properties (Theorem~\ref{thm:general_upper_bound_mod_2}). 

\begin{theorem}\label{thm:general_upper_bound_mod_2}
Let $0<u\leq 1/2$, and $G=(V,E)$ be a graph with $u$-expansion $\alpha_u(G)$, minimum degree $h \geq 1$, and such that $e(S) \leq h |S|$ for all $S \subseteq V$. 
Define   
$\delta_u:= \min\{\alpha_u(G)/h,1\}$. Then
	\[
	q^*(G) \leq 1- \inf_{0 < u \leq 1/2} \left(\frac{\delta_u}{2+\delta_u} + \frac{u}{2}\right).
	\]
\end{theorem}
This result is of a similar shape to the upper bound for $r$-regular graphs from~\cite{treelike}, where it is shown that
    \[
	q^*(G) \leq 1- \inf_{0 < u \leq 1/2} \left( \frac{\alpha_u(G)}{r} + u \right)
	\]
(and the related result in~\cite{prokhorenkova2017modularity}). The $u/2$ term in Theorem~\ref{thm:general_upper_bound_mod_2} occurs analogously to the $u$ term above, where the factor of a half comes into this term in the theorem since there is a factor of two between the minimum and average degrees instead of all degrees being equal in the regular graph.

After establishing Theorem~\ref{thm:general_upper_bound_mod_2} and the expansion results given in Section~\ref{sec:expansion}, we may obtain our modularity upper bound. The proofs of Theorem~\ref{thm:general_upper_bound_mod_2} and Corollary~\ref{cor.q_upper_bound} are given in Section~\ref{sec:proofs_mod}. 
\begin{corollary}\label{cor.q_upper_bound}
	Let $h\geq 2$ and let $G \sim G^h_n$ or $G \sim \tilde{G}^h_n$. Then whp
	\[ \q(G) \leq 0.92383.
	\] 
\end{corollary} 
We note this is the first result for $G_n^h$ and $\tilde{G}^h_n$ which bounds the modularity value strictly away from~1 whp for each $h\geq 2$.

Recall the following nice conjecture of \cite{prokhorenkova2017modularity_internet_Math}.

\begin{conjecture}
	Let $h\geq 2$ and let $G \sim G^h_n$. Then whp
	\[ \q(G) = \Theta(h^{-1/2}). \]     
\end{conjecture}
This conjecture was recently confirmed up to logarithmic factors by the second and fourth author in~\cite{rybarczyk2025modularity}. 
\begin{theorem}[Theorem 5 in \cite{rybarczyk2025modularity}]\label{thm.sqrth}
	Let $h\geq 2$ and let $G \sim G^h_n$. Then whp
	\[ \q(G) \leq f(h) \, h^{-1/2}, \] 
	where $f(h) = 3 \sqrt{2 \ln{2}} \sqrt{\ln{h}}(1+o(1))$. 
\end{theorem}
While this upper bound is stronger for large $h$, it drops below $1$ only for $h \geq 810$ (see remark in~\cite{rybarczyk2025modularity}). The upper bound for modularity obtained in this paper improves the one from \cite{rybarczyk2025modularity} for all $2 \leq h \leq 959$ and for this range of $h$ is the best known so far. 

\section{Proofs}

\subsection{Proof of Theorem \ref{thm:our_u_expansion_fussy} and Corollary \ref{cor:exp_linear_in_h}: edge expansion.} \label{sec:proofs_exp}

The following result for $\tilde{G}_n^h$ may be found in the proof of Theorem~1 of~\cite{mihail2006journal}. Since obtaining it relies only on Lemma~2 of~\cite{mihail2006journal}, which we have proven holds also for~$G^h_n$ in Appendix~\ref{app:A} (see Lemma~\ref{lem:lemma_2_us}), we have that it holds for both models. 

\begin{lemma}[\cite{mihail2006journal} and Lemma~\ref{lem:lemma_2_us}]\label{lem:existingbound_k}
	Let $h \geq 2$, $G\sim G_n^h$ or $G \sim \tilde{G}_n^h$, let $0 <\hat\alpha < (h-1)/2$, and $1 \leq k \leq n/2$. 
	Then for $n$ large enough, 
	\[
	\pr \Big(\exists S \subset V,\; |S|=k \; : \; e_G(S, \bar{S} ) \leq \hat\alpha |S| \Big) \leq \hat\alpha k \bigg(\frac{eh}{\hat\alpha}\bigg)^{2\hat\alpha k} \bigg( \frac{k}{n}\bigg)^{(h-1-2\hat\alpha)k}.
	\]
\end{lemma}

\begin{corollary}\label{cor:sum_f(k)}
Let $h \geq 2$,  $G\sim G_n^h$ or $G \sim \tilde{G}_n^h$, $0 < \hat\alpha < (h-1)/2$, and $0 < u \leq 1/2$. Define
	$f(k) = \hat\alpha k \big(\frac{eh}{\hat\alpha}\big)^{2\hat\alpha k} \big( \frac{k}{n}\big)^{(h-1-2\hat\alpha)k}$.
	Then for $n$ large enough,  
	\[
	\pr \Big(\alpha_u(G) \leq \hat\alpha  \Big) \leq \sum_{k=1}^{\lfloor un \rfloor} f(k).
	\]
\end{corollary}
\begin{proof}
	By Lemma \ref{lem:existingbound_k} we immediately get
	\[
		\pr \big(\alpha_u(G) \leq \hat\alpha  \big)  = \pr \Big(\exists S \subset V,\; 1 \leq |S|\leq un \; : \; e_G(S, \bar{S} ) \leq \hat\alpha |S| \Big) \leq \sum_{k=1}^{\lfloor un \rfloor} 
        f(k).
	\]
\end{proof}
Theorem~\ref{thm:our_u_expansion_fussy} will follow from Corollary~\ref{cor:sum_f(k)} and the following technical lemma.

\begin{lemma}[\cite{mihail2006journal}]\label{lem:functionf}
	Let $h \geq 2$, $0 < \hat\alpha < (h-1)/2$, and let $1 \leq k \leq n/2$. Define
	$f(k) = \hat\alpha k \big(\frac{eh}{\hat\alpha}\big)^{2\hat\alpha k} \big( \frac{k}{n}\big)^{(h-1-2\hat\alpha)k}$. Then for $n$ large enough, there exists $x$ with $1 \leq x \leq n/2$ such that the function $f(k)$ is decreasing for $1 \leq k \leq x$ and increasing for $x \leq k \leq n/2$. 
\end{lemma}

A version of Lemma~\ref{lem:functionf} in which the domain of $f(k)$ is narrowed to $2 \leq k \leq n/2$ can be found in the proof of Theorem~1 of~\cite{mihail2006journal}; however, we note that the proof given there applies also to the interval $1\leq k \leq n/2$. 
\begin{proof}[of Theorem~\ref{thm:our_u_expansion_fussy}]
	Note that, by Corollary~\ref{cor:sum_f(k)} and by a union bound, it is enough to ensure that there exists an integer $\ell$ such that  $f(k)=o(n^{-c})$ for each $1 \leq k \leq \ell$ and $f(k)=o(n^{-c-1})$ for each $\ell \leq k \leq un$, where $f(k)$ is as defined in Corollary~\ref{cor:sum_f(k)}. However, by Lemma~\ref{lem:functionf}, it is enough to show that $f(1) = o(n^{-c})$, $f(\ell) = o(n^{-c-1})$ and $f(un) = o(n^{-c-1})$. Set $c=((h-1)/2-\hat\alpha_u)/2$ and $\ell$ the smallest integer such that $\ell>  1/2c \,\,+ 1$. Note that our assumption $\hat\alpha_u < (h-1)/2$ ensures that $c>0$.
    We get
	\begin{eqnarray*}
	f(1) 
    & = & \hat\alpha_u \left(\frac{eh}{\hat\alpha_u}\right)^{2\hat\alpha_u} \left( \frac{1}{n}\right)^{h-1-2\hat\alpha_u}  \leq \hat\alpha_u \left(\frac{eh}{\hat\alpha_u}\right)^{2\hat\alpha_u} \left( \frac{1}{n}\right)^{2c}
     =  o(n^{-c}).
	\end{eqnarray*}
    Similarly for $\ell$, noting that $\ell \cdot 2c > 1 + 2c$,
    \begin{eqnarray*}
	f(\ell) 
    & = & \hat\alpha_u  \ell \left(\frac{eh}{\hat\alpha_u}\right)^{2\hat\alpha_u \ell} \left( \frac{ \ell}{n}\right)^{\ell( h-1-2\hat\alpha_u)} 
    \leq \hat\alpha_u  \ell  \left(\frac{eh}{\hat\alpha_u}\right)^{2\hat\alpha_u\ell} \left( \frac{ \ell}{n}\right)^{1+2c}
     = o(n^{-c-1})
	\end{eqnarray*}
    Now, let $\delta$ be given by $\hat\alpha_u = \delta h$. By the assumptions of the theorem we know that $\delta < \min\{ 1/2 - 1/2h, x\}$, where $0 < x \leq 1$ is such that
	\begin{equation} \label{eq:xh}
		\left(\frac{e}{u x}\right)^{2h x} < \frac{1}{u^{h-1}}.
	\end{equation}	
	We have 
	\[ f(un) = \delta h un \left(\left(\frac{e}{\delta u}\right)^{2\delta h} u^{h-1}\right)^{un},\]
	thus in order to get $f(un)=o(n^{-c-1})$ for some constant $c>0$ it is sufficient to prove that
	\begin{equation} \label{eq:needed<1}
		\left(\frac{e}{\delta u}\right)^{2\delta h} < \frac{1}{u^{h-1}}.
	\end{equation}
	By (\ref{eq:xh}), the fact that the function  $g_h(y)=\left(e/yu\right)^{2hy}$ is increasing for $0<y\leq1/u$, and by the fact that $0<\delta < x \leq 1 \leq 1/u$ we get
	\[
	\left(\frac{e}{\delta u}\right)^{2\delta h} < \left(\frac{e}{u x}\right)^{2h x} < \frac{1}{u^{h-1}}.
	\]
This completes the proof of Theorem~\ref{thm:our_u_expansion_fussy}.
\end{proof}

\begin{proof}[of Corollary~\ref{cor:exp_linear_in_h}]
    By Theorem~\ref{thm:our_u_expansion_fussy}, it is enough to show that $\eta=0.03418$ satisfies $(2e/\eta)^{2h\eta}<2^{h-1}$ for each $h\geq 2$. 

    Note that if $(2e/\eta)^{4\eta}<2$ then
    $\big(2e/\eta\big)^{2h\eta} < 2^{h/2} \leq 2^{h-1}$
    since $h/2\leq h-1$. Thus it is enough to check if $(2e/\eta)^{4\eta}<2$. 
	  Observe that $(2e/\eta)^{4\eta} \approx 1.99984 < 2$, and so we are done.
\end{proof}

\subsection{Proofs of Lemma~\ref{lem:us_upper_bound},
Theorem~\ref{thm:general_upper_bound_mod_2} and Corollary~\ref{cor.q_upper_bound}: modularity. 
}\label{sec:proofs_mod}

We first prove Lemma~\ref{lem:us_upper_bound}.

\begin{proof}[of Lemma~\ref{lem:us_upper_bound}]
    Write $G=(V,E)$, $n=|V|$, and $m=|E|$. Let $\cA$ be the optimal vertex partition for $G$, i.e., $q_{\cA}(G) = q^*(G)$. Note that the assumptions on the minimum and the average degree imply $m \leq hn$ and $\vol(S) \geq h|S|$ for all $S \subseteq V$. We also consider separately the edge contribution $q_{\cA}^E(G) = \sum_{S \in \cA} \frac{e(S)}{m}$ and degree tax $q_{\cA}^D(G) = \sum_{S \in \cA} \frac{\vol(S)^2}{4 m^2}$. (Note $q_{\cA}(G) = q_{\cA}^E(G) - q_{\cA}^D(G)$.) \\
	
	\noindent\textbf{Case 1.} First, assume that there exists $S \in \cA$ such that $|S| = an$ with $a \geq 1/2$. Then
	\[
	q_{\cA}^D(G) \geq \frac{\vol(S)^2}{4 m^2} \geq \frac{h^2 |S|^2}{4h^2n^2} = \frac{a^2}{4}.
	\]
	We also have %$e(S,\bar{S})/|\bar{S}| \geq h \, \delta_{1-a}$
    $e(S,\bar{S})/|\bar{S}| \geq \alpha(G)$ 
    which implies $e(S,\bar{S}) \geq \alpha(G)(1-a)n$ thus
	\[
	q_{\cA}^E(G) \leq 1 - \frac{e(S,\bar{S})}{m} \leq 1 - \frac{\alpha(G)}{h}(1-a).
	\]
	Hence
	\[
	q_{\cA}(G) \leq 1 - \frac{\alpha(G)}{h} + \frac{\alpha(G)}{h}a -  \frac{a^2}{4}.
	\]
	But if $f(a) = \frac{\alpha(G)}{h}a -  \frac{a^2}{4}$ then $f'(a) = \frac{\alpha(G)}{h} -\frac{a}{2} \leq 0$ for $1/2 \leq a \leq 1$ if $\frac{\alpha(G)}{h} \leq 1/4$. Hence
	\[
	q_{\cA}(G) \leq 1-\frac12\min\left\{\frac{\alpha(G)}{h}, \frac{1}{4}\right\} - \frac{1}{16}.
	\]
	
	\noindent\textbf{Case 2.} Now, assume that for all $S \in \cA$ we have $|S| \leq n/2$.
% 	%
	Note that the number of edges $b$ between parts of the partition $\cA$ satisfies 
	\[  2 b=\sum_{S \in \cA} e(S, \bar{S}) \geq \sum_{S \in \cA} \alpha(G) |S| = \alpha(G) n. \]
	Thus 
	\[
    q_{\cA}(G) \leq q^E_{\cA}(G)  = 1-\frac{b}{m} \leq 1 - \frac{\alpha(G) n}{2hn} = 1 - \frac{\alpha(G)}{2h}.
    \] 	
	Hence, considering both cases together,
	\[
	q_\cA(G) \leq 1- \min\left\{ \frac{\alpha(G)}{2 h}, \frac{1}{2} \min\left\{\frac{\alpha(G)}{h}, \frac{1}{4}\right\} + \frac{1}{16}\right\}.\]
	If $0 \leq \alpha(G)/h \leq 1/4$ then the bound is
	$1-\alpha(G)/2h$;
	if $1/4 \leq \alpha(G)/h \leq 3/8$ then the bound is also $1-\alpha(G)/2h$; and for $\alpha(G)/h \geq 3/8$ the bound is $1-3/16$.  This completes the proof.
\end{proof}

Before we move on to the proofs of the next results we define the negative relative modularity of vertex subsets as below.
\begin{definition}\label{defn.neg_mod}
	Let $G=(V,E)$ be a graph with at least one edge, and let $S \subseteq V$.	Define the negative relative modularity of $S$ to be 
	\[
	q_{vr}^{-}(S) = \frac{\vol(G)}{\vol(S)}\left(\frac{ e(S,\bar{S})}{2e(G)}+\frac{\vol(S)^2}{\vol(G)^2}\right). 
	\]
\end{definition}

\begin{proposition} \label{prop.bound_by_worst_set_vertex}
	Let $G=(V,E)$ be a graph with at least one edge and let $\mathcal{A}$ be any vertex partition of $G$. Then
    \[
		q_{\mathcal{A}}(G) \leq 1 - \min_{S \in \mathcal{A}} q_{vr}^{-}(S).
	\]
\end{proposition}
\begin{proof}
	Recall that the number of edges between the parts of the partition $\mathcal{A}$ equals $\sum_{S \in \mathcal{A}} (e(S,\bar{S})/2)$. Therefore
	\[
	\begin{split}
		q_{\mathcal{A}}(G) & = \sum_{S \in \mathcal{A}}\frac{e(S)}{e(G)} - \sum_{S \in \mathcal{A}}\frac{\vol(S)^2}{\vol(G)^2} = 1 - \sum_{S \in \mathcal{A}}\frac{e(S,\bar{S})}{2 e(G)} - \sum_{S \in \mathcal{A}}\frac{\vol(S)^2}{\vol(G)^2} \\
		& = 1- \sum_{S \in \mathcal{A}} \left( \frac{\vol(S)}{\vol(G)} \cdot \frac{\vol(G)}{\vol(S)}\left(\frac{ e(S,\bar{S})}{2e(G)}+\frac{\vol(S)^2}{\vol(G)^2}\right) \right) \\
		& \leq 1 - \min_{S \in \mathcal{A}} q_{vr}^{-}(S) \sum_{S \in \mathcal{A}} \frac{\vol(S)}{\vol(G)} = 1 - \min_{S \in \mathcal{A}} q_{vr}^{-}(S).
	\end{split}
	\]
\end{proof}
The following technical lemma will be needed in the proof of Theorem~\ref{thm:general_upper_bound_mod_2} -- see Appendix~\ref{app.two_lemmas} for the proof of the lemma.

\begin{lemma}\label{lem:dull}
	Let $0 <  u \leq 1/2$, and $0 \leq \delta \leq 1$. Then
	\[ \frac{\delta}{2+\delta} + \frac{u}{2} \leq \frac{\delta u}{2(1-u) +\delta u} + \frac{1-u}{2}.\]
\end{lemma}

\begin{proof}[of Theorem~\ref{thm:general_upper_bound_mod_2}]
	We write $G=(V,E)$ and $n=|V|$ and we will use Proposition \ref{prop.bound_by_worst_set_vertex}. We first show that for any $S \subseteq V$ such that $|S| =un$, where $0<u\leq 1/2$ we get
	\[
	q_{vr}^-(S) \geq \delta_u/(2+\delta_u)+u/2
	\]
	(see Defintion~\ref{defn.neg_mod}) and consider larger sets later. By the assumption of the theorem we have
	\[
		\vol(S) = 2e(S) + e(S,\bar{S}) \leq 2 h|S| + e(S,\bar{S})
	\]
	and by the definition of $u$-expansion
	\[
		e(S,\bar{S}) \geq \alpha_u(G)|S| = \delta_u h |S|.
	\]
	Therefore
	\begin{equation}\label{eq.boundedgeloss}
		\frac{e(S, \bar{S})}{\vol(S)} \geq \frac{e(S, \bar{S})}{2 h|S| + e(S,\bar{S})} 
		\geq \frac{\delta_u h|S|}{2h|S|+ \delta_u h|S|} = \frac{\delta_u}{2+\delta_u}.
	\end{equation}
    Thus for $S$ with $|S| =  un \leq n/2$ by~\eqref{eq.boundedgeloss} and since $\vol(S)\geq h|S|$,
    
    \[ q_{vr}^{-}(S) = \frac{e(S, \bar{S})}{\vol(S) } + \frac{\vol(S)}{\vol(G)} \geq \frac{\delta_u}{2+\delta_u} + \frac{h|S|}{2hn} =  \frac{\delta_u}{2+\delta_u} + \frac{u}{2}, \]
    hence we get 
   
    \begin{equation}\label{eq.usual_sets}\min_{\substack{S\subseteq V\\|S|\leq n/2}} q_{vr}^{-}(S)\geq \inf_{0 < u \leq 1/2} \Big( \frac{\delta_u}{2+\delta_u}+\frac{u}{2} \Big).
	\end{equation}
	Now consider $S \subseteq V$ such that $|S| = \bar{u}n$ and $1/2<\bar{u}\leq 1$. Note that $|\bar{S}|=(1-\bar{u})n \leq n/2$ and so, by the $u$-expansion definition, $e(S,\bar{S}) \geq \delta_{1-\bar{u}}h|\bar{S}|$. This time, similarly to~\eqref{eq.boundedgeloss},
	\[
		\frac{e(S, \bar{S})}{\vol(S)} 
		\geq \frac{e(S, \bar{S})}{2 h|S| + e(S,\bar{S})}
		\geq \frac{\delta_{1-u} h|\bar{S}|}{2h|S|+ \delta_{1-\bar{u}} h|\bar{S}|} = \frac{\delta_{1-\bar{u}} (1-\bar{u})}{2\bar{u}+\delta_{1-\bar{u}}(1-\bar{u})}
	\]
	and 
	\[ q_{vr}^{-}(S) = \frac{e(S, \bar{S})}{\vol(S) } + \frac{\vol(S)}{\vol(G)} \geq \frac{\delta_{1-\bar{u}}(1-\bar{u})}{2\bar{u}+\delta_{1-\bar{u}}(1-\bar{u})} + \frac{h|S|}{2hn} =  \frac{\delta_{1-\bar{u}}(1-\bar{u})}{2\bar{u}+\delta_{1-\bar{u}}(1-\bar{u})} + \frac{\bar{u}}{2}. \]
	Writing $s=1-\bar{u}$, for $S \subseteq V$ such that $|S|=\bar{u}n > n/2$ we get
	\[ q_{vr}^{-}(S)  \geq \frac{\delta_{s}s}{2(1-s)+\delta_{s}s} + \frac{1-s}{2} \geq \frac{\delta_s}{2+\delta_s}+ \frac{s}{2} \]
	where the second inequality followed by Lemma~\ref{lem:dull}. Thus the case when $|S| > n/2$ is covered by our existing bounds in~\eqref{eq.usual_sets} for $|S|\leq n/2$. We get
	\[
		\min_{S \subseteq V} q_{vr}^{-}(S)\geq \inf_{0 < u \leq 1/2} \Big( \frac{\delta_u}{2+\delta_u}+\frac{u}{2} \Big)
	\]
	and the conclusion follows by Proposition~\ref{prop.bound_by_worst_set_vertex}.
\end{proof}

	We may now prove our upper bound on the modularity value in the preferential attachment model.

\begin{proof}[of Corollary~\ref{cor.q_upper_bound}] 
	The first detail to note is that $G=(V,E)$ does indeed have minimum degree $h$. In both models, each vertex has $h$ mini-vertices, each of which contributes at least $1$ to the degree of that vertex.
	This is true also for the first mini-vertex in~$G \sim \tilde{G}_n^h$ since it has a loop of weight $1$ attached to it, thus contributing degree~$1$.
	
	Next, note that each edge counted in $e(S)$ (where $S \subseteq V$) had to be introduced by a mini-vertex corresponding to some vertex from $S$. Thus $e(S)$ can be at most the number of mini-vertices corresponding to $S$, i.e. $e(S) \leq h|S|$  and the assumptions of Theorem~\ref{thm:general_upper_bound_mod_2} are met.
	
    Let $0 < u\leq 1/2$ and let 
    $\delta_u:= \min\{\alpha_u(G)/h,1\}$. 
	By Theorem~\ref{thm:general_upper_bound_mod_2} we get
	\[
		q^*(G) \leq 1- \inf_{0 < u \leq 1/2} \left(\frac{\delta_u}{2+\delta_u} + \frac{u}{2}\right).
	\]
	Note that by definition (see Definition~\ref{def:uexpansion}) $\alpha_u(G)$ (and thus ${\delta}_u/(2+{\delta}_{u})$) is non-increasing in $u$.
	Hence if we have a sequence $0=u_0 \leq u_1 \leq \ldots \leq u_t = 1/2$, then
	\[
		\min_{1\leq s \leq t} \left( \frac{{\delta}_{u_s}}{2+{\delta}_{u_s}}+ \frac{u_{s-1}}{2}\right) \leq \inf_{0 < u \leq 1/2} \left( \frac{{\delta}_{u}}{2+{\delta}_{u}} + \frac{u}{2}\right)
	\]
	which implies
	\[
		q^*(G) \leq 1-	\min_{1\leq s \leq t} \left( \frac{{\delta}_{u_s}}{2+{\delta}_{u_s}}+ \frac{u_{s-1}}{2}\right)
	\]
	thus it is enough to check finitely many values to get the bound. By~Theorem~\ref{thm:our_u_expansion_fussy} we may whp bound
	$\delta_{u_s} \geq \hat{\delta}_{u_s}$ ($s \in \{0,1,\ldots,t\}$), where $\hat{\delta}_{u_s} < \min\{1/2- 1/2h, x\}$  with $0 < x \leq 1$ being such that \begin{equation}\label{eq:x_cond}(e/u_s x)^{2h x} < (1/u_s)^{h-1}.\end{equation} 
    Note that if $\eta$ satisfies~\eqref{eq:x_cond} for $h=2$, 
    then \[(e/u_s \eta)^{2h \eta} = \left( (e/u_s \eta)^{4 \eta}\right)^{h/2} < (1/u_s)^{h/2} \leq (1/u_s)^{h-1},\] i.e.~\eqref{eq:x_cond} holds for $h\geq 2$. Thus for each particular $u_s$ we may choose the same $\hat{\delta}_{u_s}< \min\{1/4, \eta \}$ 
    where  $\eta$ satisfies $(e/u_s \eta)^{4 \eta} < 1/u_s$ and for those choices whp we get
	\begin{equation} \label{eq.check_finite_2}
		q^*(G) \leq 1-	\min_{1\leq s \leq t} \left( \frac{{\hat\delta}_{u_s}}{2+{\hat\delta}_{u_s}}+ \frac{u_{s-1}}{2}\right).
	\end{equation}
	If one checks $u_0=0\leq 0.0001\leq 0.0002 \leq \ldots \leq 1/2=u_t$ and for $\hat{\delta}_{u_s}$ one chooses the largest value satisfying the inequality $\hat{\delta}_{u_s}< \min\{1/4, \eta \}$, accurate to $5$ decimal places, then one finds that 
	the minimum in~\eqref{eq.check_finite_2} achieves its value for~$s$ such that $u_s=0.0142$ by ${\hat\delta}_{u_s} = 0.14851$. This gives $\q(G)\leq 0.92383$ whp as required. 
\end{proof}

\begin{credits}
\subsubsection{\ackname} This research was funded in whole or in part by National Science Centre, Poland, grant OPUS-25 no 2023/49/B/ST6/02517. Fiona Skerman is supported by the Wallenberg AI, Autonomous Systems and Software Program (WASP) funded by the Knut and Alice Wallenberg Foundation.
\end{credits}

 \bibliographystyle{splncs04}
 \bibliography{PA_mod_note}

%\newpage
\appendix
\section{Appendix}  \label{app:A}

\subsection{Proof of Lemma 2 of~\cite{mihail2006journal} for $\mathbf{G_n^h}$}

For $n \in \mathbb{N}$ let $[n]=\{1,2,\ldots,n\}$. The following lemma was proved for $\tilde{G}_n^h$ in \cite{mihail2006journal}.
\begin{lemma}[Lemma 2, \cite{mihail2006journal}]\label{lem:lemma_2_MPS}
	Let $G\sim \tilde{G}_n^h$. For a fixed subset $S\subset \{1, \ldots, n\}$, $|S|=k$, and for a fixed subset $A \subset \{e_t : t \in [hn]\}$ with $|A|<hk$ we have
	\[\pr\Big( E_{G}(S, \; \bar{S}) = A \Big) \leq \frac{\binom{hk}{|A|}}{\binom{hn-|A|}{hk-|A|}}. \]
\end{lemma}
We will prove the result above 
holds also for 
$G\sim G_n^h$. 
Note a slightly different result was proven recently by Hazra, van der Hofstad, and Ray in \cite{hazra2023percolation} for several different preferential attachment models, including $G_n^h$ and $\tilde{G}_n^h$. 
Apart from the fact that we use slightly different notation the changes in the proof in comparison with the one from~\cite{mihail2006journal} will be very small. However, we keep here our proof for completeness. 
(The differences appear only in the sequences of inequalities (5) and (7) from \cite{mihail2006journal} which correspond to inequalities~(\ref{ineq:upp1}) and~(\ref{ineq:upp2}) in the proof below.)

\begin{lemma}\label{lem:lemma_2_us}
	Let $G\sim {G}_n^h$. For a fixed subset $S\subset \{1, \ldots, n\}$, $|S|=k$, and for a fixed subset $A \subset \{e_t : t \in [hn]\}$ with $|A|<hk$ we have
	\[
	\pr\Big( E_{G}(S, \; \bar{S}) = A \Big) \leq \frac{\binom{hk}{|A|}}{\binom{hn-|A|}{hk-|A|}}.
	\]
\end{lemma}

\newcommand{\good}{ 
	E_{G}(S, \; \bar{S})
}
\newcommand{\goodt}{
	E_{T_t}(S^{[m]}, \; \bar{S}^{[m]})
}
\newcommand{\goodtminusone}{
	E_{T_{t-1}}(S^{[m]}, \; \bar{S}^{[m]})
}
\newcommand{\inbadt}{
	\notin E_{T_t}(S^{[m]}, \; \bar{S}^{[m]})
}

\begin{proof} 
	For $w \in [hn]$ let $A^{(w)} = \{e_t \in A: t \in [w]\}$ (thus $A^{(w)}$ is the subset of edges of $T_w$ in $A$). By $S^{[m]}$ we denote the set of mini-vertices corresponding to vertices in $S$, by $\bar{S}^{[m]}$ the set of mini-vertices corresponding to vertices in $\bar{S}$. %
	Note that if $e_1 \in A$ then $\pr\Big( E_{G}(S, \; \bar{S}) = A \Big) = 0$ therefore we assume that $e_1 \notin A$. Additionally, w.l.o.g., assume that $1 \in \bar{S}$ (as we can always rename $S$ and $\bar{S}$).
	Next, let $k_1 = |\{e_t \in A: t \in S^{[m]}\}|$ 
    and $k_2 = |\{e_t \in A: t \in \bar{S}^{[m]}\}|$, thus $k_1 + k_2 = |A|$. We also define a set of arrival times of edges 
	\[
	X=\{t \in S^{[m]} : e_t \notin A\}
	\]
	and put them in a natural order, i.e.,
	\[
	X = \{t_1, t_2, \ldots, t_{hk-k_1}\}, \quad \textnormal{where} \quad  t_1 < t_2< \ldots < t_{hk-k_1}.
	\]
	Analogously
	\[
	\begin{split}
		\bar{X} & =\{t\in \bar{S}^{[m]} : e_t \notin A \} \\
		& = \{\bar{t}_1 = 1, \bar{t}_2, \bar{t}_3, \ldots, \bar{t}_{hn-hk-k_2}\}, \quad \textnormal{where} \quad \bar{t}_1 < \ldots < \bar{t}_{hn-hk-k_2}.
	\end{split}
	\]
	For $w \in [hn]$ let also $X^{(w)} = X\cap [w]$ and  $\bar{X}^{(w)} = \bar{X}\cap [w]$.
	
	Note that
	\begin{equation} \label{ineq:main}
		\begin{split}
			\pr\Big(&\good  
			=  A \Big)  =  \pr \Big( e_1 \notin E_{T_1}(S^{[m]},\bar{S}^{[m]}) \Big)\\
			& \times \prod_{e_t \in A} \pr\Big( e_t \in \goodt \;  \Big| \; \goodtminusone = A^{(t-1)} \Big)\\
			& \times \prod_{e_t \notin A \cup \{e_1\}} \pr\Big( e_t \inbadt \;  \Big| \; \goodtminusone = A^{(t-1)} \Big)\\
			\leq & \prod_{e_t \notin A \cup \{e_1\}} \pr\Big( e_t \inbadt  \;  \Big| \; \goodtminusone = A^{(t-1)} \Big) \\
			= & \prod_{i = 1}^{hk-k_1} \pr\Big( e_{t_i} \notin E_{T_{t_i}}(S^{[m]}, \bar{S}^{[m]}) \;  \Big| \; E_{T_{t_i-1}}(S^{[m]}, \bar{S}^{[m]}) = A^{(t_i-1)} \Big)\\
			& \times \prod_{i=2}^{hn-hk-k_2} \pr\Big( e_{\bar{t}_i} \notin  E_{T_{\bar{t}_i}}(S^{[m]}, \bar{S}^{[m]}) \;  \Big| \; E_{T_{\bar{t}_i-1}}(S^{[m]}, \bar{S}^{[m]}) = A^{(\bar{t_i}-1)} \Big).
		\end{split}
	\end{equation}
	By construction of the model we have
	\begin{equation} \label{ineq:S}
		\begin{split}
			\pr\Big( e_{t_i} & \notin E_{T_{t_i}}(S^{[m]}, \bar{S}^{[m]}) \;  \Big| \; E_{T_{t_i-1}}(S^{[m]}, \bar{S}^{[m]})		 = A^{(t_i-1)}\Big) \\&
			= \frac{\vol_{T_{t_i-1}}(S^{[m]})+1}{\vol(T_{t_i-1})+1} = \frac{2(i-1)+|A^{(t_i-1)}|+1}{2(|A^{(t_i-1)}|+|X^{(t_i-1)}|+|\bar{X}^{(t_i-1)}|)+1}.
		\end{split}
	\end{equation}
	Indeed, this is the probability that when mini-vertex $t_i$ arrives, it establishes an edge that will connect it with another mini-vertex from $S$, given that of the edges so far, i.e. $\{e_1, \ldots, e_{t_i-1}\}$ those in $A$ are exactly the ones that lie between $S^{[m]}$ and $\bar{S}^{[m]}$. Recall that $e_{t_i}$ can be a loop (this is where $+1$ in numerator and denominator comes from). Note also that among the edges so far, each edge from $\{e_{t_1}, e_{t_2}, \ldots, e_{t_{i-1}}\}$ contributes $2$ to $\vol_{T_{t_i-1}}(S^{[m]})$ (which gives $2(i-1)$ in numerator) and each edge from $A^{(t_i-1)}$ contributes $1$ to $\vol_{T_{t_i-1}}(S^{[m]})$ (which gives $|A^{(t_i-1)}|$ in numerator). The formula for $\vol(T_{t_i-1})$ comes simply from the fact that each edge that appeared so far contributes $2$ to the total volume of the graph.
	
	Likewise, we have
    \begin{equation} \label{ineq:Sbar}
		\begin{split}
			\pr\Big( e_{\bar{t}_i} & \notin E_{T_{\bar{t}_i}}(S^{[m]}, \bar{S}^{[m]}) \;  \Big| \; E_{T_{\bar{t}_i-1}}(S^{[m]}, \bar{S}^{[m]})		 = A^{(\bar{t}_i-1)}\Big) \\&
			= \frac{\vol_{T_{\bar{t}_i-1}}(\bar{S}^{[m]})+1}{\vol(T_{\bar{t}_i-1})+1} = \frac{2(i-1)+|A^{(\bar{t}_i-1)}|+1}{2(|A^{(\bar{t}_i-1)}|+|X^{(\bar{t}_i-1)}|+|\bar{X}^{(\bar{t}_i-1)}|)+1}.
		\end{split}
	\end{equation}

	Note that 
	\begin{equation} \label{eq:hitsAll}
		\bigcup_{t: e_t \notin A} \{ |X^{(t)}|+|\bar{X}^{(t)}|\} = \{1, \ldots, hn-|A|\}.
	\end{equation}
	To see why this set `hits' all possible values observe that if the set $E_{G}(S, \; \bar{S})$ is exactly $A$ then the value of interest increments, i.e.,  $|X^{(t)}|+|\bar{X}^{(t)}|=|X^{(t-1)}|+|\bar{X}^{(t-1)}|+1$ only if $e_t \notin A$.
	
	Let us now upperbound the expression from (\ref{ineq:S}). We get
	\begin{equation} \label{ineq:upp1}
		\begin{split}
			& \frac{2(i-1)+|A^{(t_i-1)}|+1}{2(|A^{(t_i-1)}|+|X^{(t_i-1)}|+|\bar{X}^{(t_i-1)}|)+1} 
			\overset{[1]}\leq \frac{2(i-1)+|A^{(t_i-1)}|+1}{2(|X^{(t_i-1)}|+|\bar{X}^{(t_i-1)}|)+|A^{(t_i-1)}|+1} \\ 
			& \overset{[2]}\leq \frac{2(i-1)+|A|+1}{2(|X^{(t_i-1)}|+|\bar{X}^{(t_i-1)}|)+|A|+1}
			\overset{[3]}\leq \frac{2i+|A|}{2(|X^{(t_i-1)}|+|\bar{X}^{(t_i-1)}|)+|A|+2} \\
			& = \frac{i+|A|/2}{|X^{(t_i-1)}|+|\bar{X}^{(t_i-1)}|+|A|/2+1}
			\overset{[4]}\leq \frac{i+|A|}{|X^{(t_i-1)}|+|\bar{X}^{(t_i-1)}|+|A|+1},
		\end{split}
	\end{equation}
	where the marked inequalities come from:
	\begin{enumerate}[label={[\arabic*]}]
		\item subtracting $|A^{(t_i-1)}|$ from the denominator,
		\item adding $|A|-|A^{(t_i-1)}|\geq 0$ to the numerator and the denominator,
		\item adding $1$ to the numerator and the denominator,
		\item adding $|A|/2$ to the numerator and the denominator.
	\end{enumerate}
	The same sequence of operations gives the following upper bound for the expression in~(\ref{ineq:Sbar}):
	\begin{equation} \label{ineq:upp2}
		\frac{2(i-1)+|A^{(\bar{t}_i-1)}|+1}{2(|A^{(\bar{t}_i-1)}|+|X^{(\bar{t}_i-1)}|+|\bar{X}^{(\bar{t}_i-1)}|)+1} \leq \frac{i+|A|}{|X^{(\bar{t}_i-1)}|+|\bar{X}^{(\bar{t}_i-1)}|+|A|+1}.
	\end{equation}
	Now, by (\ref{ineq:main}), (\ref{ineq:S}), (\ref{ineq:Sbar}), (\ref{eq:hitsAll}), (\ref{ineq:upp1}) and (\ref{ineq:upp2}) and assuming $|X^{(0)}|+|\bar{X}^{(0)}|=0$ we may write
	\begin{equation} \label{ineq:interm}
		\begin{split}
			\pr\Big(&\good =  A \Big) \\
			\leq & \prod_{i = 1}^{hk-k_1} \frac{i+|A|}{|X^{(t_i-1)}|+|\bar{X}^{(t_i-1)}|+|A|+1}  \prod_{i=2}^{hn-hk-k_2} \frac{i+|A|}{|X^{(\bar{t}_i-1)}|+|\bar{X}^{(\bar{t}_i-1)}|+|A|+1}\\
			= &  \prod_{i = 1}^{hk-k_1} \frac{i+|A|}{|X^{(t_i-1)}|+|\bar{X}^{(t_i-1)}|+|A|+1} \prod_{i=1}^{hn-hk-k_2} \frac{i+|A|}{|X^{(\bar{t}_i-1)}|+|\bar{X}^{(\bar{t}_i-1)}|+|A|+1}\\
			= & \frac{ \prod_{i = 1}^{hk-k_1} (i+|A|) \prod_{i=1}^{hn-hk-k_2} (i+|A|) }{\prod_{i=1}^{hn-|A|}(i+|A|)} \overset{[1]}=
			\frac{ \left( \prod_{i = 1}^{hk-k_1+|A|} i\right) \left(\prod_{i=1}^{hn-hk-k_2+|A|} i \right)}{(hn)!|A|!} \\
			\overset{[2]}= & \frac{(hk+k_2)!(hn-hk+k_1)!}{(hn)!|A|!} \\
			= & \left(\prod_{i=0}^{k_2-1}\frac{hk+k_2-i}{hn-i}\right) \left(\prod_{i=0}^{k_1-1}\frac{hn-hk+k_1-i}{hn-k_2-i}\right) \frac{(hk)!(hn-hk)!}{(hn-|A|)!|A|!}\\
			\leq & \frac{(hk)!(hn-hk)!}{(hn-|A|)!|A|!},
		\end{split}
	\end{equation}
	where the marked equalities come from 
	\begin{enumerate}[label={[\arabic*]}]
		\item multiplying the numerator and the denominator by $(|A|!)^2$,
		\item using the fact that $k_1+k_2 = |A|$
	\end{enumerate}
	and the last inequality follows by the two inequalities:
	\[
	\frac{hk+k_2-i}{hn-i} \leq 1, \quad 0 \leq i \leq k_2-1,
	\]
	\[
	\frac{hn-hk+k_1-i}{hn-k_2-i} \leq 1, \quad 0 \leq i \leq k_1-1.
	\]
	The first one is true since $k_2 \leq hn-hk = |\bar{S}^{[m]}|$ and the second one is true by the statement of the lemma $hk \geq |A| = k_1 + k_2$.
	
	Finally, multiplying the numerator and the denominator of the final expression in (\ref{ineq:interm}) by $(hk-|A|)!$ we get
	\[
	\begin{split}
		\pr\left( \good =  A \right) & \leq \frac{(hk-|A|)!(hn-hk)!}{(hn-|A|)!} \cdot \frac{(hk)!}{|A|!(hk-|A|)! } \\
		& = \frac{\binom{hk}{|A|}}{\binom{hn-|A|}{hk -|A|}}.
	\end{split}
	\]
	This completes the proof.
\end{proof}

\subsection{Proof of Lemma~\ref{lem:dull}} \label{app.two_lemmas}
\begin{proof}[of Lemma~\ref{lem:dull}]
	It is easy to check that the inequality is true for $u=1/2$. Thus for the rest of the proof assume $0<u<1/2$. Aiming for a contradiction assume that the inequality is the other way round. Then, subtracting $u/2$ from both sides and multiplying both sides by $2(2+\delta)(2(1-u)+\delta u)$ yields 
	\[
	2\delta(2(1-u) + \delta u) > 2\delta u (2+\delta) + (1-2u)(2+\delta)(2(1-u)+\delta u).
	\]
	Expanding out we get
	\[
	4\delta (1-u) + 2\delta^2 u > 4\delta u+2\delta^2 u + (1-2u)( 4(1-u)+2\delta(1-u)+ 2\delta u+\delta^2u  ).
	\]
	Rearranging gives
	\[
	4\delta (1-2u)  >  (1-2u)( 4(1-u)+2\delta + \delta^2u  ).
	\]
	We divide by $(1-2u)$ and rearrange to get
	\[
	2\delta  >   4(1-u) + \delta^2u   
	\]
	which is a contradiction since $\delta \leq 1$ and $4(1-u)\geq 2$.    
\end{proof}

\end{document}